\documentclass[12pt]{amsart}

\usepackage{amssymb,amsmath}
\usepackage{amsfonts}
\usepackage{amsthm}
\usepackage{mathrsfs}
\usepackage{graphicx}
\usepackage[T1]{fontenc}% Brzd\k{e}k

\theoremstyle{plain} 
\newtheorem{theorem}{Theorem}

\theoremstyle{definition} 
\newtheorem{definition}[theorem]{Definition}

\newtheorem{example}[theorem]{Example}
\newcommand{\R}{\ensuremath{\mathbb{R}}}
\newcommand{\Z}{\ensuremath{\mathbb{Z}}}
\newcommand{\T}{\ensuremath{\mathbb{T}}}
\newcommand{\N}{\ensuremath{\mathbb{N}}}

\numberwithin{equation}{section}
\numberwithin{theorem}{section}

\small\normalsize

\begin{document}

\title[Hyers--Ulam Stability on Time Scales]{Hyers--Ulam Stability for Discrete Time Scale with Two Step Sizes}
\author[Anderson]{Douglas R. Anderson} 
\address{Department of Mathematics \\
         Concordia College \\
         Moorhead, MN 56562 USA}
\email{andersod@cord.edu}
%\author[Onitsuka]{Masakazu Onitsuka}
%\address{Okayama University of Science \\
%Department of Applied Mathematics \\
%Okayama, 700-0005, Japan}
%\email{onitsuka@xmath.ous.ac.jp}


\begin{abstract} 
We clarify the Hyers--Ulam stability (HUS) of certain first-order linear constant coefficient dynamic equations on time scales, in the case of a specific time scale with two alternating step sizes, where the exponential function changes sign.
\end{abstract}

\maketitle\thispagestyle{empty}

%%%%%%%%%%%%%%%%% 
%               % 
% SECTION      % 
% First-Order   % 
%%%%%%%%%%%%%%%%%

\section{Hyers--Ulam Stability on a Specific Time Scale}

%%%%%%%%%%%%%%%%%%%%%%
%                    %
% T=P_{\alpha,\beta} %
%                    %
%%%%%%%%%%%%%%%%%%%%%%

\begin{definition}
Let $\T$ be a time scale and $\lambda\in\R$ be a constant. We say that the eigenvalue equation
\begin{equation}\label{maineq}
 x^\Delta(t) = \lambda x(t), \quad \lambda\in\R, \quad t\in\T
\end{equation}
has Hyers--Ulam stability (HUS) on $\T$ iff there exists a constant $K>0$ with the following property. For arbitrary $\varepsilon>0$, if a function $\phi:\T\rightarrow\R$ satisfies $|\phi^\Delta(t)-\lambda\phi(t)|\le\varepsilon$ for all $t\in\T^\kappa$, then there exists a solution $x:\T\rightarrow\R$ of \eqref{maineq} such that $|\phi(t)-x(t)|\le K\varepsilon$ for all $t\in\T$. Such a constant $K$ is called an HUS constant for \eqref{maineq} on $\T$.
\end{definition}

\begin{definition}
Eigenvalue equation \eqref{maineq} is regressive iff $1+\lambda\mu(t)\ne 0$ for all $t\in\T^{\kappa}$, and \eqref{maineq} is positively regressive iff $1+\lambda\mu(t) > 0$ for all $t\in\T^{\kappa}$.
\end{definition}

One question that arises following the above definitions is, if \eqref{maineq} has HUS, what is the minimum HUS constant?
When $\T=\R$, Onitsuka and Shoji \cite{onitsuka2} show that the minimum HUS constant for \eqref{maineq} is $1/\lambda$ if $\lambda>0$ and $1/|\lambda|$ if $\lambda<0$. When $\T=h\Z$, Onitsuka \cite{onitsuka1} shows that the minimum HUS constant for \eqref{maineq} is $1/\lambda$ if $\lambda>0$; $1/|\lambda|$ if $-1/h<\lambda<0$; $1/(\lambda+2/h)$ if $-2/h<\lambda<-1/h$; and $1/|\lambda+2/h|$ if $\lambda<-2/h$. Note that there is no HUS for \eqref{maineq} for $\lambda=0$ on arbitrary $\T$, or for $\lambda=-2/h$ when $\T=h\Z$. 

It is striking to see this more complicated list in the uniformly discrete case $\T=h\Z$, particularly for negative eigenvalues $\lambda$. Due to this observation, it is surmised that finding minimum HUS constants on arbitrary time scales for all values of $\lambda\in\R$, particularly for all $\lambda<0$ when \eqref{maineq} is not positively regressive, may prove to be a prohibitive task. Indeed, on general time scales the current situation is as follows. Andr\'{a}s and M\'{e}sz\'{a}ros \cite[Theorem 2.5]{andras} highlight three cases when considering HUS for \eqref{maineq}, namely
\begin{enumerate}
 \item[S1] $|e_\lambda(t,t_0)|$ and $\int_{t_0}^{t}|e_\lambda(t,\sigma(s))|\Delta s$ are bounded on $[t_0,\infty)_\T$; 
 \item[S2] $\lim_{t\rightarrow\infty}|e_\lambda(t,t_0)|=\infty$ and $\int_{t}^{\infty}|e_{\lambda}(t_0,\sigma(s))|\Delta s<\infty$ for all $t\in[t_0,\infty)_{\T}$;
 \item[S3] $|e_\lambda(t,t_0)|$ is bounded on $[t_0,\infty)_{\T}$ and $\lim_{t\rightarrow\infty}\int_{t_0}^{t}|e_\lambda(s,t_0)|\Delta s=\infty$.
\end{enumerate}
In the case of S1, they prove that \eqref{maineq} has HUS with HUS constant 
\begin{equation}\label{andrasK}
 K=\sup|e_\lambda(t,t_0)|+\sup \int_{t_0}^{t}|e_\lambda(t,\sigma(s))|\Delta s, 
\end{equation}
and in the case of S2 with $\lambda\ne 0$, they prove that \eqref{maineq} has HUS with HUS constant $K=1/|\lambda|$, but they were not necessarily seeking to prove these HUS constants $K$ were minimal. That paper \cite{andras} introduced case S3 but left its impact on questions of HUS unaddressed. Recently, Anderson and Onitsuka \cite{andon} used the methods of \cite{andras, onitsuka1, onitsuka2} to find the following: If $\lambda\ne 0$ and $1+\lambda\mu(t)>0$ for all $t\in\T^{\kappa}$, then \eqref{maineq} has HUS with minimum HUS constant $K=1/|\lambda|$. Also, partially addressing the case of S3 above, they found that if there exist constants $0<m<M$ such that
\begin{equation}\label{e-bdd}
 0 < m \le |e_{\lambda}(t,t_0)| \le M, \qquad \forall\;t\in\T,
\end{equation}
then \eqref{maineq} does not have HUS. It remains open as to what happens when $1+\lambda\mu(t)<0$, and whether the HUS constant $K$ given in \eqref{andrasK} in the case of S1 is the minimum HUS constant. Shen \cite{shen} partly deals with the situation when $1+\lambda\mu(t)<0$, but only for finite intervals. 

To further explore the case of S1 above and eigenvalues $\lambda$ such that $1+\lambda\mu(t)<0$, in this work we consider a discrete time scale with alternating graininess function. In particular, for the two step sizes $\alpha,\beta>0$ with $\alpha\ne\beta$, let
\[ \T:=\ensuremath{\mathbb{T}}_{\alpha,\beta}=\{0, \alpha, (\alpha+\beta), (\alpha+\beta)+\alpha, 2(\alpha+\beta), 2(\alpha+\beta)+\alpha, 3(\alpha+\beta),\cdots\}. \]
Then for $t\in\T$ and $k\in\N_0=\{0,1,2,3,4,\cdots\}$ we have 
\[ \mu(t)=\begin{cases} \alpha:& t=k(\alpha+\beta) \\ \beta:& t=k(\alpha+\beta)+\alpha, \end{cases} \]
and for $\lambda\in\R\backslash\{-1/\alpha,-1/\beta\}$ the time scales exponential function $e_{\lambda}(t,0)$ is given by
\begin{equation}\label{pabexp0} 
 e_{\lambda}(t,0) = \begin{cases} \left[(1+\lambda\alpha)(1+\lambda\beta)\right]^{\frac{t}{\alpha+\beta}}, & t=k(\alpha+\beta), \\ 
 \left[(1+\lambda\alpha)(1+\lambda\beta)\right]^{\frac{t-\alpha}{\alpha+\beta}}(1+\lambda\alpha), &t=k(\alpha+\beta)+\alpha. \end{cases}
\end{equation}
By exponential function we mean the unique solution to \eqref{maineq} satisfying initial condition $x(0)=1$.

Our method will be as follows. We will apply the techniques developed for $\T=\R$ in \cite{onitsuka2} and $\T=h\Z$ in \cite{onitsuka1} to $\T=\T_{\alpha,\beta}$ defined above, and then compare the HUS constants thus derived with that given by \cite{andras} in \eqref{andrasK}. Note that in the following theorem only case I identifies a minimum HUS constant. In cases G and H, however, if we let the two step sizes satisfy $\alpha=\beta=h$, then noting that $\lambda<0$ the constant given in the theorem reduces to
\[ \frac{\max\left\{\frac{1}{\beta}-\frac{1}{\alpha}-\lambda,\; \frac{1}{\alpha}-\frac{1}{\beta}-\lambda\right\}}{|\lambda|\left|\lambda+\frac{1}{\alpha}+\frac{1}{\beta}\right|} = \frac{-\lambda}{|\lambda|\left|\lambda+\frac{2}{h}\right|} = \frac{1}{\left|\lambda+2/h\right|}, \]
which is indeed the minimum constant for $\T=h\Z$ from \cite{onitsuka1}, as mentioned earlier.

% Theorem %

\begin{theorem}\label{mainthm}
Consider \eqref{maineq} for $\lambda\in\R$. If $\alpha^2+\beta^2-6\alpha\beta\ge 0$, set
\[ \lambda^+:=\frac{-\alpha-\beta+\sqrt{\alpha^2+\beta^2-6\alpha\beta}}{2\alpha\beta}, \quad \lambda^-:=\frac{-\alpha-\beta-\sqrt{\alpha^2+\beta^2-6\alpha\beta}}{2\alpha\beta}. \]
Let $\varepsilon>0$ be given. Then we have the following cases.
\begin{enumerate}
\item[A.] Let $\alpha\ge(3+2\sqrt{2})\beta$, and let $\frac{-1}{\beta}<\lambda<\lambda^-$ or $\lambda^+<\lambda<\frac{-1}{\alpha}$. Then \eqref{maineq} has HUS with an HUS constant $\frac{|\lambda+\frac{1}{\alpha}-\frac{1}{\beta}|}{(\lambda-\lambda^+)(\lambda-\lambda^-)}$.
\item[B.] Let $\alpha > (3+2\sqrt{2})\beta$, and let $\lambda^-<\lambda<\lambda^+$. Then \eqref{maineq} has HUS with an HUS constant $\left|\frac{\lambda+\frac{1}{\alpha}-\frac{1}{\beta}}{(\lambda-\lambda^+)(\lambda-\lambda^-)}\right|$.
\item[C.] Let $\beta < \alpha < (3+2\sqrt{2})\beta$, and let $\frac{-1}{\beta}<\lambda<\frac{-1}{\alpha}$. Then \eqref{maineq} has HUS with an HUS constant $\left|\frac{\lambda+\frac{1}{\alpha}-\frac{1}{\beta}}{\frac{2+\alpha\lambda+\beta\lambda+\alpha\beta\lambda^2}{\alpha\beta}}\right|$.
\item[D.] Let $0 < \alpha \le (3-2\sqrt{2})\beta$, and let $\frac{-1}{\alpha}<\lambda<\lambda^-$ or $\lambda^+<\lambda<\frac{-1}{\beta}$. Then \eqref{maineq} has HUS with an HUS constant $\frac{\frac{1}{\alpha}-\frac{1}{\beta}-\lambda}{(\lambda-\lambda^+)(\lambda-\lambda^-)}$.
\item[E.] Let $0 < \alpha < (3-2\sqrt{2})\beta$, and let $\lambda^-<\lambda<\lambda^+$. Then \eqref{maineq} has HUS with an HUS constant $\frac{\frac{1}{\alpha}-\frac{1}{\beta}-\lambda}{|(\lambda-\lambda^+)(\lambda-\lambda^-)|}$.
\item[F.] Let $(3-2\sqrt{2})\beta < \alpha < \beta$, and let $\frac{-1}{\alpha}<\lambda<\frac{-1}{\beta}$. Then \eqref{maineq} has HUS with an HUS constant $\left|\frac{\frac{1}{\alpha}-\frac{1}{\beta}-\lambda}{\frac{2+\alpha\lambda+\beta\lambda+\alpha\beta\lambda^2}{\alpha\beta}}\right|$.
\item[G.] Let $\frac{-1}{\alpha}-\frac{1}{\beta} < \lambda < \min\left\{\frac{-1}{\alpha},\frac{-1}{\beta}\right\}$. Then \eqref{maineq} has HUS with an HUS constant $\frac{\max\left\{\frac{1}{\beta}-\frac{1}{\alpha}-\lambda,\; \frac{1}{\alpha}-\frac{1}{\beta}-\lambda\right\}}{|\lambda|\left(\lambda+\frac{1}{\alpha}+\frac{1}{\beta}\right)}$.
\item[H.] Let $\lambda<\frac{-1}{\alpha}-\frac{1}{\beta}$. Then \eqref{maineq} has HUS with an HUS constant $\frac{\max\left\{\frac{1}{\beta}-\frac{1}{\alpha}-\lambda,\; \frac{1}{\alpha}-\frac{1}{\beta}-\lambda\right\}}{|\lambda|\left|\lambda+\frac{1}{\alpha}+\frac{1}{\beta}\right|}$.
\item[I.] If $\max\left\{\frac{-1}{\alpha},\frac{-1}{\beta}\right\}<\lambda<0$ or $\lambda>0$, then \eqref{maineq} has HUS with minimum HUS constant $1/|\lambda|$.
\item[J.] If $\lambda=0$, $\lambda=\lambda^+$, $\lambda=\lambda^-$, or $\lambda=\frac{-1}{\alpha}-\frac{1}{\beta}$, then \eqref{maineq} does not have HUS.
\item[K.] If $\lambda=\frac{-1}{\alpha}$ or $\lambda=-\frac{1}{\beta}$, then \eqref{maineq} does not exist as a first-order dynamic equation.
\end{enumerate}
\end{theorem}

\begin{proof}
We proceed through the various cases.

A. If $\alpha\ge(3+2\sqrt{2})\beta$, and either $\frac{-1}{\beta}<\lambda<\lambda^-$ or $\lambda^+<\lambda<\frac{-1}{\alpha}$, then $-1<(1+\lambda\alpha)(1+\lambda\beta)<0$. Let $\varepsilon>0$ be given.
Suppose $\phi:\T\rightarrow\R$ satisfies 
\begin{equation}\label{phidelA} 
 |\phi^\Delta(t)-\lambda\phi(t)|\le \varepsilon, \quad\forall t\in\T. 
\end{equation}
We will show that if $x$ solves \eqref{maineq} with 
\begin{equation}\label{x0A}
 |\phi(0)-x(0)|<\frac{\varepsilon(\lambda+\frac{1}{\alpha}+\frac{1}{\beta})}{(\lambda-\lambda^+)(\lambda-\lambda^-)}, 
\end{equation}
then
\begin{equation}\label{xtA}
 |\phi(t)-x(t)| < \frac{\varepsilon\left|\lambda+\frac{1}{\alpha}-\frac{1}{\beta}\right|}{(\lambda-\lambda^+)(\lambda-\lambda^-)}. 
\end{equation}
Let
\begin{equation}\label{e1}
 E_1(t):=\begin{cases} (-1)^{\frac{t}{\alpha+\beta}}\left(\lambda+\frac{1}{\alpha}+\frac{1}{\beta}\right) &: t=k(\alpha+\beta), \\ 
 (-1)^{\frac{t-\alpha}{\alpha+\beta}}\left(\lambda+\frac{1}{\alpha}-\frac{1}{\beta}\right) &: t=k(\alpha+\beta)+\alpha,\end{cases}
\end{equation}
and define the functions $u$ and $v$ on $\T$ via
\[ u(t):=\left(\phi(t)+\frac{\varepsilon E_1(t)}{(\lambda-\lambda^+)(\lambda-\lambda^-)}\right)\frac{1}{e_{\lambda}(t,0)} \quad\text{and}\quad v(t):=\left(\phi(t)-\frac{\varepsilon E_1(t)}{(\lambda-\lambda^+)(\lambda-\lambda^-)}\right)\frac{1}{e_{\lambda}(t,0)} \]
for $t\in\T$ using \eqref{e1}. Then
\begin{equation}\label{phiA}
 \phi(t)=e_{\lambda}(t,0)u(t)-\frac{\varepsilon E_1(t)}{(\lambda-\lambda^+)(\lambda-\lambda^-)}=e_{\lambda}(t,0)v(t)+\frac{\varepsilon E_1(t)}{(\lambda-\lambda^+)(\lambda-\lambda^-)},
\end{equation}
and since $E_1(t)$ and $e_{\lambda}(t,0)$ have the same sign for all $t\in\T$ in this case, we have
\begin{equation}\label{uvA}
 u(t)=v(t)+\frac{2\varepsilon}{(\lambda-\lambda^+)(\lambda-\lambda^-)}\left|\frac{E_1(t)}{e_{\lambda}(t,0)}\right|, 
\end{equation}
meaning $u(t)>v(t)$ for all $t\in\T$. Additionally, for $t=k(\alpha+\beta)$ or $t=k(\alpha+\beta)+\alpha$,
\begin{equation}\label{udelA}
 u^\Delta(t) = \Big(\big(\phi^\Delta(t)-\lambda\phi(t)\big)+\varepsilon(-1)^{k+1}\Big)\frac{1}{e^\sigma_{\lambda}(t,0)}; 
\end{equation}
$(1+\lambda\alpha)<0<(1+\lambda\beta)$ in this case imply that $e^\sigma_{\lambda}(t,0)$ and $(-1)^{k+1}$ have the same sign, hence this and  \eqref{phidelA} result in the inequality
\begin{equation}\label{udelposA} 0 \le u^\Delta(t) \le \frac{2\varepsilon}{|e^\sigma_{\lambda}(t,0)|}. 
\end{equation}
Analogously
\begin{equation}\label{vdelposA} 
\frac{-2\varepsilon}{|e^\sigma_{\lambda}(t,0)|} \le v^\Delta(t) \le 0. 
\end{equation}
Thus $u$ non-decreasing and $v$ non-increasing, with $u(t)>v(t)$ means that
\begin{equation}\label{3.5A}
 v(t) \le v(0) < u(0) \le u(t), \quad \forall t\in\T. 
\end{equation}
Assume $x$ solves \eqref{maineq} and satisfies initial condition \eqref{x0A}. Clearly $x(t)=x(0)e_{\lambda}(t,0)$ for all $t\in\T$, and using the initial condition and \eqref{phiA} we have that
\[ v(0) < x(0) < u(0). \] 
If $k$ is odd and $t=k(\alpha+\beta)$, then \eqref{pabexp0}, \eqref{e1}, and \eqref{phiA} yield
\begin{eqnarray*}
 \phi(t)-x(t) & = & [u(t)- x(0)][(1+\lambda\alpha)(1+\lambda\beta)]^k - \frac{\varepsilon (-1)^{k}\left(\lambda+\frac{1}{\alpha}+\frac{1}{\beta}\right) }{(\lambda-\lambda^+)(\lambda-\lambda^-)} \\
              &\le& [x(0)-u(0)]\left|(1+\lambda\alpha)(1+\lambda\beta)\right|^k + \frac{\varepsilon\left(\lambda+\frac{1}{\alpha}+\frac{1}{\beta}\right) }{(\lambda-\lambda^+)(\lambda-\lambda^-)} \\
							& < & \frac{\varepsilon\left(\lambda+\frac{1}{\alpha}+\frac{1}{\beta}\right) }{(\lambda-\lambda^+)(\lambda-\lambda^-)};
\end{eqnarray*}
similarly,
\begin{eqnarray*}
 \phi(t)-x(t) & = & [v(t)- x(0)][(1+\lambda\alpha)(1+\lambda\beta)]^k + \frac{\varepsilon (-1)^{k}\left(\lambda+\frac{1}{\alpha}+\frac{1}{\beta}\right) }{(\lambda-\lambda^+)(\lambda-\lambda^-)} \\
              &\ge& [x(0)-v(0)]\left|(1+\lambda\alpha)(1+\lambda\beta)\right|^k - \frac{\varepsilon\left(\lambda+\frac{1}{\alpha}+\frac{1}{\beta}\right) }{(\lambda-\lambda^+)(\lambda-\lambda^-)} \\
							& > & \frac{-\varepsilon\left(\lambda+\frac{1}{\alpha}+\frac{1}{\beta}\right) }{(\lambda-\lambda^+)(\lambda-\lambda^-)}. 
\end{eqnarray*}
If $k$ is odd and $t=k(\alpha+\beta)+\alpha$, then \eqref{pabexp0}, \eqref{e1}, and \eqref{phiA} yield
\begin{eqnarray*}
 \phi(t)-x(t) & = & [u(t)- x(0)][(1+\lambda\alpha)(1+\lambda\beta)]^k(1+\lambda\alpha) - \frac{\varepsilon (-1)^{k}\left(\lambda+\frac{1}{\alpha}-\frac{1}{\beta}\right) }{(\lambda-\lambda^+)(\lambda-\lambda^-)} \\
              &\ge& [u(0)-x(0)]\left|(1+\lambda\alpha)(1+\lambda\beta)\right|^k|1+\lambda\alpha| + \frac{\varepsilon\left(\lambda+\frac{1}{\alpha}-\frac{1}{\beta}\right) }{(\lambda-\lambda^+)(\lambda-\lambda^-)} \\
							& > & \frac{\varepsilon\left(\lambda+\frac{1}{\alpha}-\frac{1}{\beta}\right) }{(\lambda-\lambda^+)(\lambda-\lambda^-)};
\end{eqnarray*}
similarly,
\begin{eqnarray*}
 \phi(t)-x(t) & = & [v(t)- x(0)][(1+\lambda\alpha)(1+\lambda\beta)]^k(1+\lambda\alpha) + \frac{\varepsilon (-1)^{k}\left(\lambda+\frac{1}{\alpha}-\frac{1}{\beta}\right) }{(\lambda-\lambda^+)(\lambda-\lambda^-)} \\
              &\le& [v(0)-x(0)]\left|(1+\lambda\alpha)(1+\lambda\beta)\right|^k|1+\lambda\alpha| + \frac{\varepsilon\left|\lambda+\frac{1}{\alpha}-\frac{1}{\beta}\right|}{(\lambda-\lambda^+)(\lambda-\lambda^-)} \\
							& < & \frac{\varepsilon\left|\lambda+\frac{1}{\alpha}-\frac{1}{\beta}\right|}{(\lambda-\lambda^+)(\lambda-\lambda^-)}.
\end{eqnarray*}
If $k$ is even and $t=k(\alpha+\beta)$, then \eqref{pabexp0}, \eqref{e1}, and \eqref{phiA} yield
\begin{eqnarray*}
 \phi(t)-x(t) & = & [u(t)- x(0)]\left|(1+\lambda\alpha)(1+\lambda\beta)\right|^k - \frac{\varepsilon\left(\lambda+\frac{1}{\alpha}+\frac{1}{\beta}\right) }{(\lambda-\lambda^+)(\lambda-\lambda^-)} \\
              &\ge& [u(0)-x(0)]\left|(1+\lambda\alpha)(1+\lambda\beta)\right|^k - \frac{\varepsilon\left(\lambda+\frac{1}{\alpha}+\frac{1}{\beta}\right) }{(\lambda-\lambda^+)(\lambda-\lambda^-)} \\
							& > & \frac{-\varepsilon\left(\lambda+\frac{1}{\alpha}+\frac{1}{\beta}\right) }{(\lambda-\lambda^+)(\lambda-\lambda^-)};
\end{eqnarray*}
similarly,
\begin{eqnarray*}
 \phi(t)-x(t) & = & [v(t)- x(0)]\left|(1+\lambda\alpha)(1+\lambda\beta)\right|^k + \frac{\varepsilon\left(\lambda+\frac{1}{\alpha}+\frac{1}{\beta}\right) }{(\lambda-\lambda^+)(\lambda-\lambda^-)} \\
              &\le& [v(0)-x(0)]\left|(1+\lambda\alpha)(1+\lambda\beta)\right|^k + \frac{\varepsilon\left(\lambda+\frac{1}{\alpha}+\frac{1}{\beta}\right)}{(\lambda-\lambda^+)(\lambda-\lambda^-)} \\
							& < & \frac{\varepsilon\left(\lambda+\frac{1}{\alpha}-\frac{1}{\beta}\right)}{(\lambda-\lambda^+)(\lambda-\lambda^-)}.
\end{eqnarray*}
If $k$ is even and $t=k(\alpha+\beta)+\alpha$, then \eqref{pabexp0}, \eqref{e1}, and \eqref{phiA} yield
\begin{eqnarray*}
 \phi(t)-x(t) & = & [-u(t) + x(0)]|(1+\lambda\alpha)(1+\lambda\beta)|^k|1+\lambda\alpha| - \frac{\varepsilon\left(\lambda+\frac{1}{\alpha}-\frac{1}{\beta}\right) }{(\lambda-\lambda^+)(\lambda-\lambda^-)} \\
              &\le& [x(0)-u(0)]|(1+\lambda\alpha)(1+\lambda\beta)|^k|1+\lambda\alpha| + \frac{\varepsilon\left|\lambda+\frac{1}{\alpha}-\frac{1}{\beta}\right|}{(\lambda-\lambda^+)(\lambda-\lambda^-)} \\
							& < & \frac{\varepsilon\left|\lambda+\frac{1}{\alpha}-\frac{1}{\beta}\right|}{(\lambda-\lambda^+)(\lambda-\lambda^-)};
\end{eqnarray*}
similarly,
\begin{eqnarray*}
 \phi(t)-x(t) & = & [-v(t) + x(0)]|(1+\lambda\alpha)(1+\lambda\beta)|^k|1+\lambda\alpha| + \frac{\varepsilon\left(\lambda+\frac{1}{\alpha}-\frac{1}{\beta}\right) }{(\lambda-\lambda^+)(\lambda-\lambda^-)} \\
              &\ge& [x(0)-v(0)]|(1+\lambda\alpha)(1+\lambda\beta)|^k|1+\lambda\alpha| + \frac{\varepsilon\left(\lambda+\frac{1}{\alpha}-\frac{1}{\beta}\right)}{(\lambda-\lambda^+)(\lambda-\lambda^-)} \\
							& > & \frac{\varepsilon\left(\lambda+\frac{1}{\alpha}-\frac{1}{\beta}\right)}{(\lambda-\lambda^+)(\lambda-\lambda^-)}.
\end{eqnarray*}
All of these cases lead to the conclusion that
\[ |\phi(t)-x(t)| < \frac{\varepsilon}{(\lambda-\lambda^+)(\lambda-\lambda^-)}\begin{cases} \left(\lambda+\frac{1}{\alpha}+\frac{1}{\beta}\right) &: t=k(\alpha+\beta), \\ \left|\lambda+\frac{1}{\alpha}-\frac{1}{\beta}\right| &: t=k(\alpha+\beta)+\alpha, \end{cases} \]
for all $t\in\T$ and all $k\in\N_0$. As $\left|\lambda+\frac{1}{\alpha}-\frac{1}{\beta}\right|>\lambda+\frac{1}{\alpha}-\frac{1}{\beta}$ in case A, \eqref{xtA} holds. \hfill$\diamondsuit$

B. If $\alpha > (3+2\sqrt{2})\beta$, and $\frac{-1}{\beta} < \lambda^- < \lambda < \lambda^+ < \frac{-1}{\alpha}$, then $(1+\lambda\alpha)(1+\lambda\beta)<-1$. Assume \eqref{phidelA}, \eqref{e1}, and \eqref{phiA}. Then \eqref{uvA}, \eqref{udelA}, \eqref{udelposA}, and \eqref{vdelposA} hold, where in \eqref{uvA} we note that now in this case 
\[ (\lambda-\lambda^+)<0<(\lambda-\lambda^-), \]
so that $u(t)<v(t)$ for all $t\in\T$. Moreover,
\[ \left|\frac{E_1(t)}{e_{\lambda}(t,0)}\right| = \begin{cases} \frac{\lambda+\frac{1}{\alpha}+\frac{1}{\beta}}{|(1+\lambda\alpha)(1+\lambda\beta)|^{\frac{t}{\alpha+\beta}}} &: t=k(\alpha+\beta) \\ \frac{\left|\lambda+\frac{1}{\alpha}-\frac{1}{\beta}\right|}{|(1+\lambda\alpha)(1+\lambda\beta)|^{\frac{t-\alpha}{\alpha+\beta}}|1+\lambda\alpha|} &: t=k(\alpha+\beta)+\alpha \end{cases} \]
with $|(1+\lambda\alpha)(1+\lambda\beta)|>1$ for all $t\in\T$ implies that
\[ \lim_{t\rightarrow\infty}\left|\frac{E_1(t)}{e_{\lambda}(t,0)}\right| = 0. \] 
Consequently in this case we have for all $t\in\T$ that
\begin{equation}\label{3.5B}
 u(0)\le u(t)\le \lim_{t\rightarrow\infty}u(t)=\lim_{t\rightarrow\infty}\frac{\phi(t)}{e_{\lambda}(t,0)}=\lim_{t\rightarrow\infty}v(t)\le v(t)\le v(0).
\end{equation}
Consider the function
\[ x(t):=\left(\lim_{t\rightarrow\infty}\frac{\phi(t)}{e_{\lambda}(t,0)}\right)e_{\lambda}(t,0), \quad t\in\T, \]
which is a well-defined solution of \eqref{maineq}. This together with \eqref{phiA} and \eqref{3.5B} yields
\begin{eqnarray}
 \phi(t)-x(t) &=& \left(u(t)-\lim_{t\rightarrow\infty}u(t)\right) e_\lambda(t,0)-\frac{\varepsilon E_1(t)}{(\lambda-\lambda^+)(\lambda-\lambda^-)} \label{phixub} \\
  &=& \left(v(t)-\lim_{t\rightarrow\infty}v(t)\right) e_\lambda(t,0)+\frac{\varepsilon E_1(t)}{(\lambda-\lambda^+)(\lambda-\lambda^-)}. \label{phixvb}
\end{eqnarray}
If $k$ is odd and $t=k(\alpha+\beta)$, then using \eqref{phixub} and checking signs we have
\[ \phi(t)-x(t) > \frac{-\varepsilon\left(\lambda+\frac{1}{\alpha}+\frac{1}{\beta}\right)}{\left|(\lambda-\lambda^+)(\lambda-\lambda^-)\right|}, \]
using \eqref{phixvb} and checking signs we have
\[ \phi(t)-x(t) < \frac{\varepsilon\left(\lambda+\frac{1}{\alpha}+\frac{1}{\beta}\right)}{\left|(\lambda-\lambda^+)(\lambda-\lambda^-)\right|}. \]
If $k$ is even and $t=k(\alpha+\beta)$, then using \eqref{phixub} and checking signs we have
\[ \phi(t)-x(t) < \frac{\varepsilon\left(\lambda+\frac{1}{\alpha}+\frac{1}{\beta}\right)}{\left|(\lambda-\lambda^+)(\lambda-\lambda^-)\right|}, \]
using \eqref{phixvb} and checking signs we have
\[ \phi(t)-x(t) > \frac{-\varepsilon\left(\lambda+\frac{1}{\alpha}+\frac{1}{\beta}\right)}{\left|(\lambda-\lambda^+)(\lambda-\lambda^-)\right|}. \]
If $k$ is odd and $t=k(\alpha+\beta)+\alpha$, then using \eqref{phixub} and checking signs we have
\[ \phi(t)-x(t) < \frac{\varepsilon\left|\lambda+\frac{1}{\alpha}-\frac{1}{\beta}\right|}{\left|(\lambda-\lambda^+)(\lambda-\lambda^-)\right|}, \]
using \eqref{phixvb} and checking signs we have
\[ \phi(t)-x(t) > \frac{-\varepsilon\left|\lambda+\frac{1}{\alpha}-\frac{1}{\beta}\right|}{\left|(\lambda-\lambda^+)(\lambda-\lambda^-)\right|}. \]
If $k$ is even and $t=k(\alpha+\beta)+\alpha$, then using \eqref{phixub} and checking signs we have
\[ \phi(t)-x(t) > \frac{-\varepsilon\left|\lambda+\frac{1}{\alpha}-\frac{1}{\beta}\right|}{\left|(\lambda-\lambda^+)(\lambda-\lambda^-)\right|}, \]
using \eqref{phixvb} and checking signs we have
\[ \phi(t)-x(t) < \frac{\varepsilon\left|\lambda+\frac{1}{\alpha}-\frac{1}{\beta}\right|}{\left|(\lambda-\lambda^+)(\lambda-\lambda^-)\right|}. \]
As in case A, all of these cases lead to the conclusion that
\[ |\phi(t)-x(t)| < \varepsilon\left|\frac{\lambda+\frac{1}{\alpha}-\frac{1}{\beta}}{(\lambda-\lambda^+)(\lambda-\lambda^-)}\right| \]
for all $t\in\T$ and all $k\in\N_0$. \hfill$\diamondsuit$

C. If $\beta < \alpha < (3+2\sqrt{2})\beta$, and $\frac{-1}{\beta} < \lambda < \frac{-1}{\alpha}$, then $-1<(1+\lambda\alpha)(1+\lambda\beta)<0$. As in case A, assume \eqref{phidelA} and \eqref{e1}. Since $\lambda^+$ and $\lambda^-$ are not real in this case, but
\[ (\lambda-\lambda^-)(\lambda-\lambda^+) = \frac{2+\alpha\lambda+\beta\lambda+\alpha\beta\lambda^2}{\alpha\beta}\in\R, \]
define the functions $u$ and $v$ on $\T$ via
\[ u(t):=\left(\phi(t)+\frac{\varepsilon E_1(t)}{\left(\frac{2+\alpha\lambda+\beta\lambda+\alpha\beta\lambda^2}{\alpha\beta}\right)}\right)\frac{1}{e_{\lambda}(t,0)} \quad\text{and}\quad v(t):=\left(\phi(t)-\frac{\varepsilon E_1(t)}{\left(\frac{2+\alpha\lambda+\beta\lambda+\alpha\beta\lambda^2}{\alpha\beta}\right)}\right)\frac{1}{e_{\lambda}(t,0)} \]
for $t\in\T$ using \eqref{e1}. Then with minor modifications \eqref{phiA}, \eqref{uvA}, \eqref{udelA}, \eqref{udelposA}, and \eqref{vdelposA} hold and this case C is thus akin to case A. \hfill$\diamondsuit$

D. If $0 < \alpha \le (3-2\sqrt{2})\beta$, and $\frac{-1}{\alpha}<\lambda<\lambda^-$ or $\lambda^+<\lambda<\frac{-1}{\beta}$, then $-1<(1+\lambda\alpha)(1+\lambda\beta)<0$. Suppose $\phi:\T\rightarrow\R$ satisfies \eqref{phidelA} for all $t\in\T$.
We will show that if $x$ solves \eqref{maineq} with 
\begin{equation}\label{x0D}
 |\phi(0)-x(0)|<\frac{\varepsilon\left(\frac{1}{\alpha}-\frac{1}{\beta}-\lambda\right)}{(\lambda-\lambda^+)(\lambda-\lambda^-)}, 
\end{equation}
then
\begin{equation}\label{xtD}
 |\phi(t)-x(t)| < \frac{\varepsilon\left(\frac{1}{\alpha}-\frac{1}{\beta}-\lambda\right)}{(\lambda-\lambda^+)(\lambda-\lambda^-)}. 
\end{equation}
Let
\begin{equation}\label{e2}
 E_2(t):=\begin{cases} (-1)^{\frac{t}{\alpha+\beta}}\left(\frac{1}{\alpha}-\frac{1}{\beta}-\lambda\right) &: t=k(\alpha+\beta), \\ 
 (-1)^{\frac{t-\alpha}{\alpha+\beta}}\left(\frac{1}{\alpha}+\frac{1}{\beta}+\lambda\right) &: t=k(\alpha+\beta)+\alpha,\end{cases}
\end{equation}
and define the functions $u$ and $v$ on $\T$ via
\[ u(t):=\left(\phi(t)+\frac{\varepsilon E_2(t)}{(\lambda-\lambda^+)(\lambda-\lambda^-)}\right)\frac{1}{e_{\lambda}(t,0)} \quad\text{and}\quad v(t):=\left(\phi(t)-\frac{\varepsilon E_2(t)}{(\lambda-\lambda^+)(\lambda-\lambda^-)}\right)\frac{1}{e_{\lambda}(t,0)} \]
for $t\in\T$ using \eqref{e2}. Then
\begin{equation}\label{phiD}
 \phi(t)=e_{\lambda}(t,0)u(t)-\frac{\varepsilon E_2(t)}{(\lambda-\lambda^+)(\lambda-\lambda^-)}=e_{\lambda}(t,0)v(t)+\frac{\varepsilon E_2(t)}{(\lambda-\lambda^+)(\lambda-\lambda^-)},
\end{equation}
and since $E_2(t)$ and $e_{\lambda}(t,0)$ have the same sign for all $t\in\T$ in this case, we have
\begin{equation}\label{uvD}
 u(t)=v(t)+\frac{2\varepsilon}{(\lambda-\lambda^+)(\lambda-\lambda^-)}\left|\frac{E_2(t)}{e_{\lambda}(t,0)}\right|, 
\end{equation}
meaning $u(t)>v(t)$ for all $t\in\T$. Additionally, for $t=k(\alpha+\beta)$ or $t=k(\alpha+\beta)+\alpha$,
\begin{equation}\label{udelD}
 u^\Delta(t) = \left(\big(\phi^\Delta(t)-\lambda\phi(t)\big)+\varepsilon \begin{cases} (-1)^{k} &: t=k(\alpha+\beta) \\  (-1)^{k+1} &: t=k(\alpha+\beta)+\alpha\end{cases}\right)\frac{1}{e^\sigma_{\lambda}(t,0)}; 
\end{equation}
$(1+\lambda\beta)<0<(1+\lambda\alpha)$ in this case imply that 
\[ e^\sigma_{\lambda}(t,0) \quad\text{and} \quad \begin{cases} (-1)^{k} &: t=k(\alpha+\beta) \\  (-1)^{k+1} &: t=k(\alpha+\beta)+\alpha\end{cases} \] 
have the same sign, hence this and \eqref{phidelA} result in the inequality
\begin{equation}\label{udelposD} 0 \le u^\Delta(t) \le \frac{2\varepsilon}{|e^\sigma_{\lambda}(t,0)|}. 
\end{equation}
Analogously
\begin{equation}\label{vdelposD} 
\frac{-2\varepsilon}{|e^\sigma_{\lambda}(t,0)|} \le v^\Delta(t) \le 0. 
\end{equation}
Thus $u$ non-decreasing and $v$ non-increasing, with $u(t)>v(t)$ means that \eqref{3.5A} holds for all $t\in\T$. Assume $x$ solves \eqref{maineq} and satisfies initial condition \eqref{x0D}. Clearly $x(t)=x(0)e_{\lambda}(t,0)$ for all $t\in\T$, and using the initial condition and \eqref{phiD} we have that
\[ v(0) < x(0) < u(0). \] 
We can now proceed, as in the proof of case A, on estimating $\phi(t)-x(t)$ for the various cases as needed. This results in the conclusion that
\[ |\phi(t)-x(t)| < \frac{\varepsilon}{(\lambda-\lambda^+)(\lambda-\lambda^-)}\begin{cases} \left(\frac{1}{\alpha}-\frac{1}{\beta}-\lambda\right) &: t=k(\alpha+\beta), \\ \left(\frac{1}{\alpha}+\frac{1}{\beta}+\lambda\right) &: t=k(\alpha+\beta)+\alpha, \end{cases} \]
for all $t\in\T$ and all $k\in\N_0$. As $\left(\frac{1}{\alpha}-\frac{1}{\beta}-\lambda\right)>\left(\frac{1}{\alpha}+\frac{1}{\beta}+\lambda\right)$ in case D, \eqref{xtD} holds. \hfill$\diamondsuit$

E. If $0 < \alpha < (3-2\sqrt{2})\beta$, and $\frac{-1}{\alpha}<\lambda^-<\lambda<\lambda^+<\frac{-1}{\beta}$, then $(1+\lambda\alpha)(1+\lambda\beta)<-1$. Assume $\phi:\T\rightarrow\R$ satisfies \eqref{phidelA}. As in the proof of case D using \eqref{e2}, we write $\phi$ as in \eqref{phiD} for the same functions $u,v$, and \eqref{uvD} holds. Since $(\lambda-\lambda^+)<0<(\lambda-\lambda^-)$ and $|(1+\lambda\alpha)(1+\lambda\beta)|>1$, we see that $u(t)\le v(t)$ with $u$ non-decreasing and $v$ non-increasing, and
\[ \lim_{t\rightarrow\infty} \left|\frac{E_2(t)}{e_{\lambda}(t,0)}\right|=0. \] 
Therefore \eqref{3.5B} holds, and we can proceed as in the proof of case B, which yields
\[ |\phi(t)-x(t)| < \frac{\varepsilon\left(\frac{1}{\alpha}-\frac{1}{\beta}-\lambda\right)}{|(\lambda-\lambda^+)(\lambda-\lambda^-)|} \]
for all $t\in\T$. \hfill$\diamondsuit$

F. If $(3-2\sqrt{2})\beta < \alpha < \beta$ and $\frac{-1}{\alpha}<\lambda<\frac{-1}{\beta}$, then $-1<(1+\lambda\alpha)(1+\lambda\beta)<0$. The proof is similar to the proof of case D and is omitted.  \hfill$\diamondsuit$

G. If $\frac{-1}{\alpha}-\frac{1}{\beta} < \lambda < \min\left\{\frac{-1}{\alpha},\frac{-1}{\beta}\right\}$, then $0<(1+\lambda\alpha)(1+\lambda\beta)<1$. Assume $\phi:\T\rightarrow\R$ satisfies \eqref{phidelA}. Let
\begin{equation}\label{e3}
 E_3(t):=\begin{cases} \left(\frac{1}{\alpha}-\frac{1}{\beta}-\lambda\right) &: t=k(\alpha+\beta), \\ 
 \left(\frac{1}{\alpha}-\frac{1}{\beta}+\lambda\right) &: t=k(\alpha+\beta)+\alpha,\end{cases}
\end{equation}
and let $u,v:\T\rightarrow\R$ be given in terms of $\phi$ and $E_3$ via
\begin{equation}\label{phiG}
 \phi(t)=e_{\lambda}(t,0)u(t)+\frac{\varepsilon E_3(t)}{\lambda(\lambda+\frac{1}{\alpha}+\frac{1}{\beta})}=e_{\lambda}(t,0)v(t)-\frac{\varepsilon E_3(t)}{\lambda(\lambda+\frac{1}{\alpha}+\frac{1}{\beta})}.
\end{equation}
Since $E_3(t)$ and $e_{\lambda}(t,0)$ have the same sign for all $t\in\T$ in this case G, we have
\begin{equation}\label{uvG}
 u(t)=v(t)+\frac{2\varepsilon}{|\lambda|(\lambda+\frac{1}{\alpha}+\frac{1}{\beta})}\left|\frac{E_3(t)}{e_{\lambda}(t,0)}\right|, 
\end{equation}
meaning $u(t)>v(t)$ for all $t\in\T$. Additionally, for $t=k(\alpha+\beta)$ or $t=k(\alpha+\beta)+\alpha$,
\begin{equation}\label{udelG}
 u^\Delta(t) = \left(\big(\phi^\Delta(t)-\lambda\phi(t)\big)+\varepsilon\begin{cases} -1 &: t=k(\alpha+\beta) \\ 1 &: t=k(\alpha+\beta)+\alpha \end{cases}\right)\frac{1}{e^\sigma_{\lambda}(t,0)}; 
\end{equation}
as $e^\sigma_{\lambda}(t,0)$ and $\begin{cases} -1 &: t=k(\alpha+\beta) \\ 1 &: t=k(\alpha+\beta)+\alpha \end{cases}$ have the same sign, hence this and \eqref{phidelA} result in the inequality \eqref{udelposA}, and \eqref{3.5A} holds again. Suppose $x$ solves \eqref{maineq} with
\[ |\phi(0)-x(0)| < \frac{\varepsilon(\frac{1}{\alpha}-\frac{1}{\beta}-\lambda)}{|\lambda|\left(\lambda+\frac{1}{\alpha}+\frac{1}{\beta}\right)}. \]
Then $x(t)=x(0)e_\lambda(t,0)$, $v(0)<x(0)<u(0)$, and using \eqref{phiG} and the various cases we see that
\begin{equation}\label{phixG} 
 |\phi(t)-x(t)| < \frac{\varepsilon}{|\lambda|\left(\lambda+\frac{1}{\alpha}+\frac{1}{\beta}\right)}\begin{cases} \left(\frac{1}{\beta}-\frac{1}{\alpha}-\lambda\right) &: \beta<\alpha, \\ \left(\frac{1}{\alpha}-\frac{1}{\beta}-\lambda\right) &: \alpha<\beta, \end{cases} 
\end{equation}  
for all $t\in\T$. \hfill$\diamondsuit$

H. If $\lambda < \frac{-1}{\alpha}-\frac{1}{\beta}$, then $(1+\lambda\alpha)(1+\lambda\beta)>1$. As in the proof of case G, consider $E_3$ in \eqref{e3} and $\phi$ in \eqref{phiG}. Unlike \eqref{uvG}, however, we have in this case that
\begin{equation}\label{uvH}
 u(t)=v(t)-\frac{2\varepsilon}{|\lambda|\left|\lambda+\frac{1}{\alpha}+\frac{1}{\beta}\right|}\left|\frac{E_3(t)}{e_{\lambda}(t,0)}\right|, 
\end{equation}
meaning $u(t)<v(t)$ for all $t\in\T$, with $u$ non-decreasing and $v$ non-increasing. As $|E_3(t)/e_{\lambda}(t,0)|\rightarrow 0$ as $t\rightarrow\infty$, we see that \eqref{3.5B} holds. Proceeding in a way similar to the proofs of cases B and G, we see that \eqref{phixG} holds for all $t\in\T$. \hfill$\diamondsuit$

I. This case is proven in \cite[Theorem 3.7]{andon}. \hfill$\diamondsuit$

J. Let $e_\lambda(t,0)$ be given as in \eqref{pabexp0}. If $\lambda=0$, then $e_0(t,0)\equiv 1$. If $\lambda=\lambda^+$, 
then 
\[ e_{\lambda^+}(t,0) = \begin{cases} (-1)^{\frac{t}{\alpha+\beta}}, & t=k(\alpha+\beta), \\ 
 (-1)^{\frac{t-\alpha}{\alpha+\beta}}(1+\lambda^+\alpha), &t=k(\alpha+\beta)+\alpha, \end{cases} \] 
and if $\lambda=\lambda^-$, then 
\[ e_{\lambda^-}(t,0) = \begin{cases} (-1)^{\frac{t}{\alpha+\beta}}, & t=k(\alpha+\beta), \\ 
 (-1)^{\frac{t-\alpha}{\alpha+\beta}}(1+\lambda^-\alpha), &t=k(\alpha+\beta)+\alpha. \end{cases} \] 
Finally, if $\lambda=\frac{-1}{\alpha}-\frac{1}{\beta}$, then 
\[ e_{\left(\frac{-1}{\alpha}-\frac{1}{\beta}\right)}(t,0) = \begin{cases} 1, & t=k(\alpha+\beta), \\ 
 \frac{-\alpha}{\beta}, &t=k(\alpha+\beta)+\alpha. \end{cases} \]
Thus in each of these cases, \eqref{e-bdd} holds with appropriately chosen bounds, and so by \cite[Theorem 3.10 (ii)]{andon}, equation \eqref{maineq} does not have HUS. \hfill$\diamondsuit$

K. If $\lambda=\frac{-1}{\alpha}$ or $\lambda=-\frac{1}{\beta}$, then $1+\lambda\mu(t)=0$, so that \eqref{maineq} is not regressive. %\hfill$\diamondsuit$
\end{proof}

\section{Comparison of Hyers--Ulam Constants}

In this section we will compare some of the HUS constants found in Theorem \ref{mainthm} with the constant \eqref{andrasK} from \cite{andras}.
First, let us calculate $K$ given in \eqref{andrasK} using the expression for the exponential function in \eqref{pabexp0}. In particular, we calculate the following using rules from \cite[Section 1.4]{bp}. 
If $t=k(\alpha+\beta)$, then 
\begin{eqnarray*}
 \int_{0}^{t}|e_{\lambda}(t,\sigma(s))|\Delta s &=& \int_{0}^{t}|(1+\lambda\alpha)(1+\lambda\beta)|^{\frac{t-\sigma(s)}{\alpha+\beta}}\Delta s \\
 &=& \left(\int_{0}^{\alpha} + \int_{\alpha}^{(\alpha+\beta)} +\cdots+ \int_{(k-1)(\alpha+\beta)+\alpha}^{k(\alpha+\beta)}\right) |(1+\lambda\alpha)(1+\lambda\beta)|^{\frac{t-\sigma(s)}{\alpha+\beta}}\Delta s \\
 &=& \alpha|(1+\lambda\alpha)(1+\lambda\beta)|^{\frac{t-\alpha}{\alpha+\beta}} 
     +\beta|(1+\lambda\alpha)(1+\lambda\beta)|^{\frac{t-(\alpha+\beta)}{\alpha+\beta}} \\
 & & +\alpha|(1+\lambda\alpha)(1+\lambda\beta)|^{\frac{t-(\alpha+\beta)-\alpha}{\alpha+\beta}}+\cdots
		 +\beta|(1+\lambda\alpha)(1+\lambda\beta)|^{\frac{t-k(\alpha+\beta)}{\alpha+\beta}} \\
 &=& \alpha|(1+\lambda\alpha)(1+\lambda\beta)|^{\frac{t-\alpha}{\alpha+\beta}}\sum_{j=0}^{k-1}|(1+\lambda\alpha)(1+\lambda\beta)|^{-j} \\
 & & +\beta|(1+\lambda\alpha)(1+\lambda\beta)|^{\frac{t-(\alpha+\beta)}{\alpha+\beta}}\sum_{j=0}^{k-1}|(1+\lambda\alpha)(1+\lambda\beta)|^{-j} \\
 &=& \frac{|(1+\lambda\alpha)(1+\lambda\beta)|^k-1}{|(1+\lambda\alpha)(1+\lambda\beta)|-1}\left(\alpha|(1+\lambda\alpha)(1+\lambda\beta)|^{\frac{\beta}{\alpha+\beta}}+\beta\right).
\end{eqnarray*}
If $t=k(\alpha+\beta)+\alpha$, then 
\begin{eqnarray*}
 \int_{0}^{t}|e_{\lambda}(t,\sigma(s))|\Delta s 
 &=& \left(\int_{0}^{\alpha} + \int_{\alpha}^{(\alpha+\beta)} +\cdots+ \int_{k(\alpha+\beta)}^{k(\alpha+\beta)+\alpha}\right) |(1+\lambda\alpha)(1+\lambda\beta)|^{\frac{t-\sigma(s)}{\alpha+\beta}}\Delta s \\
 &=& \alpha|(1+\lambda\alpha)(1+\lambda\beta)|^{\frac{t-\alpha}{\alpha+\beta}} 
     +\beta|(1+\lambda\alpha)(1+\lambda\beta)|^{\frac{t-(\alpha+\beta)}{\alpha+\beta}} \\
 & & +\cdots+\alpha|(1+\lambda\alpha)(1+\lambda\beta)|^{\frac{t-k(\alpha+\beta)-\alpha}{\alpha+\beta}} \\
 &=& \alpha|(1+\lambda\alpha)(1+\lambda\beta)|^{\frac{t-\alpha}{\alpha+\beta}}\sum_{j=0}^{k}|(1+\lambda\alpha)(1+\lambda\beta)|^{-j} \\
 & & +\beta|(1+\lambda\alpha)(1+\lambda\beta)|^{\frac{t-(\alpha+\beta)}{\alpha+\beta}}\sum_{j=0}^{k-1}|(1+\lambda\alpha)(1+\lambda\beta)|^{-j} \\
 &=& \frac{\beta|(1+\lambda\alpha)(1+\lambda\beta)|^{\frac{\alpha}{\alpha+\beta}}(|(1+\lambda\alpha)(1+\lambda\beta)|^k-1)+\alpha(|(1+\lambda\alpha)(1+\lambda\beta)|^{k+1}-1)}{|(1+\lambda\alpha)(1+\lambda\beta)|-1}.
\end{eqnarray*}
Suppose $0<|(1+\lambda\alpha)(1+\lambda\beta)|<1$. Then 
\[ \lim_{k\rightarrow\infty} |(1+\lambda\alpha)(1+\lambda\beta)|^k=0, \]
and we can summarize the above as
\[ \lim_{t\rightarrow\infty}\int_{0}^{t}|e_{\lambda}(t,\sigma(s))|\Delta s = \begin{cases} \frac{\beta+\alpha|(1+\lambda\alpha)(1+\lambda\beta)|^{\frac{\beta}{\alpha+\beta}}}{1-|(1+\lambda\alpha)(1+\lambda\beta)|} &: t=k(\alpha+\beta), \\
 &  \\
\frac{\alpha+\beta|(1+\lambda\alpha)(1+\lambda\beta)|^{\frac{\alpha}{\alpha+\beta}}}{1-|(1+\lambda\alpha)(1+\lambda\beta)|} &: t=k(\alpha+\beta)+\alpha. \end{cases} \]
In light of this fact, for comparison's sake the HUS constant $K$ in \eqref{andrasK} is
\begin{equation}\label{roughK}
 K=\begin{cases} \sup|(1+\lambda\alpha)(1+\lambda\beta)|^{\frac{t}{\alpha+\beta}} + \frac{\beta+\alpha|(1+\lambda\alpha)(1+\lambda\beta)|^{\frac{\beta}{\alpha+\beta}}}{1-|(1+\lambda\alpha)(1+\lambda\beta)|} &: t=k(\alpha+\beta), \\
 &  \\
\sup|(1+\lambda\alpha)(1+\lambda\beta)|^{\frac{t-\alpha}{\alpha+\beta}}|1+\lambda\alpha| + \frac{\alpha+\beta|(1+\lambda\alpha)(1+\lambda\beta)|^{\frac{\alpha}{\alpha+\beta}}}{1-|(1+\lambda\alpha)(1+\lambda\beta)|} &: t=k(\alpha+\beta)+\alpha. \end{cases}
\end{equation}

\begin{example}
We will compare the HUS constant from Theorem \ref{mainthm} case A, with the HUS constant $K$ in \eqref{roughK}. Let $\alpha=6$, $\beta=1$, and $\lambda=-1/5$. Then $\lambda^+=-1/2$ and $\lambda^-=-2/3$, and 
\[ \frac{-1}{2} < \frac{-1}{5} < \frac{-1}{6} \quad\implies\quad \lambda^+ < \lambda < \frac{-1}{\alpha}. \]
Then the HUS constant from case A in Theorem \ref{mainthm} is
\[ \frac{|\lambda+\frac{1}{\alpha}-\frac{1}{\beta}|}{(\lambda-\lambda^+)(\lambda-\lambda^-)} = \frac{155}{21}\approx 7.38, \]
while $K$ in \eqref{roughK} is
\[ K=\begin{cases} 1 + 1+\frac{25}{21}\left(1+6\left(\frac{2}{5}\right)^{2/7}\right)\approx 7.688 &: t=k(\alpha+\beta), \\
 &  \\
1/5 + 1+\frac{25}{21}\left(6+\left(\frac{2}{5}\right)^{12/7}\right)\approx 7.59 &: t=k(\alpha+\beta)+\alpha, \end{cases} \]
so the HUS constant in Theorem \ref{maineq} A is better. However, if $\lambda=-4/5$, then 
\[ \frac{-1}{\beta} < \lambda < \lambda^- \]
in case A, and we have
\[ 40.8333 \quad\text{versus}\quad K=\begin{cases} 29.2055 \\ 32.0933, \end{cases} \]
so that $K$ in \eqref{roughK} from \cite{andras} is better.
\end{example}

\begin{example}
We will compare the HUS constant from Theorem \ref{mainthm} case C, with the HUS constant $K$ in \eqref{roughK}. Let $\alpha=3$, $\beta=1$, and $\lambda=-1/2$. Then
\[ -1 < \frac{-1}{2} < \frac{-1}{3} \quad\implies\quad \frac{-1}{\beta} < \lambda < \frac{-1}{\alpha}, \]
and the HUS constant from Theorem \ref{mainthm} C is
\[ \frac{14}{3}\approx 4.66, \]
while $K$ in \eqref{roughK} is
\[ K=\begin{cases} 1 + \frac{4}{3} \left(1 + \frac{3}{\sqrt{2}}\right) \approx 5.16 &: t=k(\alpha+\beta), \\
 &  \\
\frac{1}{2} + \frac{4}{3} \left(3 + \frac{1}{2 \sqrt{2}}\right)\approx 4.97 &: t=k(\alpha+\beta)+\alpha, \end{cases} \]
so the HUS constant in Theorem \ref{maineq} C is better. However, if $\lambda=-4/5$, then we have
\[ 6.111 \quad\text{versus}\quad K=\begin{cases} 5.42 \\ 6.101. \end{cases} \]
\end{example}

\begin{example}
We will compare the HUS constant from Theorem \ref{mainthm} D, with the HUS constant $K$ in \eqref{roughK}. Let $\alpha=1/10$, $\beta=1$, and $\lambda=-1.2$. Then we have
\[ 1.238 \quad\text{versus}\quad K=\begin{cases} 2.238 \\ 2.037, \end{cases} \]
so the HUS constant in D is smaller. However, if $\lambda=-9.2$, then we have
\[ 5.29 \quad\text{versus}\quad K=\begin{cases} 4.10\\ 3.168. \end{cases} \]
\end{example}

\begin{example}
We will compare the HUS constant from Theorem \ref{mainthm} G, with the HUS constant $K$ in \eqref{roughK}. 
Let $\alpha=1$, $\beta=1/2$, and $\lambda=-2.5$. Then we have
\[ 2.8 \quad\text{versus}\quad K=\begin{cases} 2.95 \\ 3.52, \end{cases} \]
but if $\lambda=-2.9$, then it is 
\[ 13.45 \quad\text{versus}\quad K=\begin{cases} 10.99 \\ 11.9. \end{cases} \]
\end{example}

As we can observe following these examples, it remains an open question as to how to find the minimum HUS constant, given that \eqref{maineq} is HUS and $1+\lambda\mu(t)<0$, even for this relatively straightforward time scale, much less for arbitrary time scales.
 
%\section*{Acknowledgements}
%The second author was supported by JSPS KAKENHI Grant Number JP17K14226.

%%%%%%%%%%%%%%%% 
% Bibliography % 
%%%%%%%%%%%%%%%%

\end{document}